\documentclass[preprint,12pt]{elsarticle}
\usepackage{amsmath,amssymb,amsthm}
\usepackage{amsfonts}
\usepackage{color}
\usepackage{multirow}
\usepackage{amstext}

\newtheorem{theo}{Theorem}[section]
\newtheorem{lema}[theo]{Lemma}

\newtheorem{prop}[theo]{Proposition}
\theoremstyle{definition}

\newtheorem{defi}[theo]{Definition}

\def\Har{\operatorname{Har}}
\def\Harhat{\operatorname{\widehat Har}}
\def\H{\mathbb{H}}
\def\C{\mathbb{C}}
\def\M{\mathcal{M}}
\def\N{\mathcal{N}}

\def\Sc{\operatorname{Sc}}
\def\Vec{\operatorname{Vec}}
\def\Dbar{\overline{\partial}}

\def\Im{\operatorname{Im}}
\def\Re{\operatorname{Re}}
\def\s{\raisebox{-1.3ex}{\rule{0pt}{4ex}}}
\def\st{\raisebox{-1.3ex}{\rule{0pt}{5ex}}}

\newcommand \R      {\mathbb R}
\newcommand \be     {\begin{equation}}
\newcommand \ee     {\end{equation}}
\newcommand{\fun}[3]{#1_{#2}^{#3}}
\newcommand{\funhat}[3]{\fun{\widehat{#1}}{#2}{#3}}
\newcommand{\funbar}[3]{\fun{\overline{#1}}{#2}{#3}}
\newcommand{\harorigq}[1]{\funhat{U}{#1}{}}      
\newcommand{\harorig}[2]{\fun{U}{#1}{#2}}   
\newcommand{\harbaseq}[1]{\funhat{V}{#1}{}} 
\newcommand{\harbase}[2]{\fun{V}{#1}{#2}}   
\newcommand{\monog}[2]{\fun{X}{#1}{#2}}
\newcommand{\antimonog}[2]{\funbar{X}{#1}{#2}}
\newcommand{\ambi}[2]{Y_{#1}^{#2}}

\newcommand{\contra}[2]{\fun{Z}{#1}{#2}}

\newcommand{\coefi}[1]{\alpha_{#1}}
\newcommand{\coefii}[1]{\beta_{#1}}
\newcommand{\coefiii}[1]{\gamma_{#1}}
\newcommand{\coefiv}[1]{a_{#1}}
\newcommand{\coefv}[2]{b_{#1}^{#2}}

\DeclareMathOperator{\arcsinh}{arcsinh}

\usepackage{amssymb}

\journal{}

\begin{document}

\begin{frontmatter}



\title{Contragenic Functions on Spheroidal Domains}


\author{R. Garc{\'i}a Ancona$^{a}$, J. Morais$^{b}$ and R. Michael Porter$^{a}$}

\address{$^{a}$Department of Mathematics, CINVESTAV-Quer\'etaro,
  Apdo.\ Postal 1-798, Arteaga 5, 76000 Santiago de Quer\'etaro, Qro.,
  Mexico. \\
$^{b}$Department of Mathematics, ITAM, Rio Hondo \#1, Col. Progreso Tizap\'an, M\'exico, DF 01080, M\'exico. E-mail: joao.morais@itam.mx}


\begin{abstract}
We construct bases of polynomials for the spaces of
  square-integra\-ble harmonic functions which are orthogonal to the
  monogenic and antimonogenic $\R^3$-valued functions defined in a
  prolate or oblate spheroid.
\end{abstract}


\end{frontmatter}


\section{Introduction}

The theory of holomorphic functions in Clifford algebras, and in particular quaternions, is quite extensive \cite{Brackx,Gurlebeck2,Sudbery}. In recent years the corresponding
results for quaternionic-differentiable functions defined in domains of $\R^3$ have been developed \cite{Bock2,Cacao1,Cacao2,Cacao3,Cacao4,Delanghe2009,Leutwiler,Morais7,Morais8,MoPe,MoK,Cuiming} with a view to making this theory applicable to physical systems. In particular, a function from $\R^3$ to $\R^3$ is quaternion-holomorphic (monogenic) precisely when it satisfies the Riesz system of partial differential equations \cite{Delanghe2007,Morais,Morais3,Morais4,MoraisAG}.

Spheres are commonly used as the reference domain for modeling physical problems.  However, in many cases, a spheroidal domain may offer a better approximation to reality. Here we study the natural
basis of harmonic polynomials in $\Har_2(\Omega)=L_2(\Omega) \cap \Har(\Omega)$ where $\Omega$ is a prolate or oblate spheroid. Our attention is directed to what are now known as contragenic functions, which are orthogonal to all monogenic and antimonogenic $L_2$ harmonic functions.  Previous to \cite{Alvarez} the existence of such functions was not suspected.  It is necessary to understand the contragenic functions in order to be able to consider the ``monogenic part'' of a given harmonic function. Contragenicity, in contrast to harmonicity and monogenicity, is not a local property, since it depends on the domain under consideration.

In Section \ref{sec:harmonics} the spheroidal harmonics are defined following \cite{Garabedian}, with a rescaling factor which permits including the sphere as a limit of both the prolate and oblate
cases, combined into a single one-parameter family.  The spheroidal monogenic polynomials are calculated in Section \ref{sec:monogenics}, and new explicit formulas for their nonscalar parts are obtained in
terms of the spheroidal harmonics. A basis for the space of functions obtained by summing a monogenic function with an antimonogenic function is written out.  All of these orthogonal bases are composed
of elements parametrized by the shape of the corresponding spheroid, and their norms are calculated explicitly.  In the final section we produce an orthogonal basis for the contragenic spheroidal
polynomials.


\section{Spheroidal harmonics \label{sec:harmonics} }

Analysis of harmonic and monogenic functions on spheroids (cf.\ \cite{Garabedian,Hobson,Morais2,Morais5,Morais6,MoPeK,MoNgK,Nguyen}
typically separates the prolate and oblate cases, parametrized in
their respective confocal families
\[ \{x\in\R^3 |\  \frac{x_0^2}{\cosh^2\alpha} + \frac{x_1^2 +
x_2^2}{\sinh^2\alpha} < 1\}, \]
\[    \{x\in\R^3 |\   \frac{x_0^2}{\sinh^2\alpha} + \frac{x_1^2 +
x_2^2}{\cosh^2\alpha} < 1\},
\]
for $\alpha>0$. These domains do not include the case of a Euclidean
ball (where harmonic analysis originated), but they become rounder as
they degenerate with $\alpha\to\infty$.  Thus we prefer to combine
them into a single family
\begin{equation}\label{eq:Omegamu}
  \Omega_\mu =
  \{x \in \R^3 |\  x_0^2 + \frac{x_1^2 + x_2^2}{e^{2\nu}} < 1 \},
\end{equation}
where $\nu\in\R$ is arbitrary, and $\mu=(1-e^{2\nu})^{\frac{1}{2}}$ by
convention is in the interval $(0,1)$ when $\nu<0$ (prolate spheroid), and in
$i\R^+$ when $\nu>0$ (oblate spheroid); the intermediate value
$\nu=0$, $\mu=0$ gives the unit ball $\Omega_0=\{x\colon\ |x|^2<1\}$.
The convenience of the parameter $\mu$ will become evident later.
Here we note that in the prolate case, we obtain $\Omega_\mu$ by
setting $e^\nu=\tanh\alpha$ and rescaling $x$ by a factor of
$\mu^{-1}$, while for the oblate case, we set $e^\nu=\coth\alpha$ and
rescale by a factor of $(\mu/i)^{-1}$.

The spheroidal harmonics $\harorig{n,m}{\pm}[\mu]$ of \cite{Garabedian}
are constructed as follows.  Suppose first that $\nu<0$. For this case
we use coordinates $(u ,v,\phi)$ defined on the prolate spheroid by
\begin{equation}\label{eq:prolatecoords}
  x_0 = \mu \cos u \cosh v, \quad x_1 = \mu \sin u \sinh v \cos \phi, \quad
  x_2 = \mu \sin u \sinh v \sin \phi,
\end{equation}
where $u\in[0,\pi]$, $v\in[0,\mbox{arctanh}\,e^\nu]$,
$\phi\in[0,2\pi]$, and $\mu>0$. Then for $x\in\Omega_\mu$ we define
\begin{align} \label{eq:prolatepreharmonics}
  \harorig{n,m}{\pm}[\mu](x)
  =& \, \coefi{n,m}\mu^n
  P_{n}^{m}(\cos u) P_{n}^{m}(\cosh v) \Phi^{\pm}_m(\phi),
\end{align}
 where \begin{equation}\label{eq:rescale}
\coefi{n,m} = \frac{(n-m)!} {(2n-1)!!}
\end{equation}
(with use of the symbol $n!!=\prod_{k=0}^{\lceil n/2 \rceil-1} (n-2k)$
for the double factorial) and we have written
\begin{equation}\label{eq:defPhi}
\Phi_m^+(\phi)=\cos (m\phi), \quad \Phi_m^-(\phi)=\sin (m\phi)
\end{equation}
in order to unify the notation for the odd and even functions. Here $P_n^m$
denotes the associated Legendre function of the first kind \cite[Ch. III]{Hobson}, of degree
$n$ and order $m$.

Observe that $\cos u = 2x_0/\omega$ and $\cosh v = \omega/(2\mu)$,
where
\begin{equation}\label{eq:omega}
  \omega=\sqrt{(x_0+\mu)^2+x_1^2+x_2^2}+\sqrt{(x_0-\mu)^2+x_1^2+x_2^2}
\end{equation}
is positive.  The oblate case $\nu>0$ is obtained from
this by analytic continuation, thinking of $\mu\in i\R^+$ as being
boundary values of the first quadrant in complex plane.  The terms
$\zeta=(|x|^2+\mu^2)+2x_0\mu$ and $\overline{\zeta}$ inside the
radicals in \eqref{eq:omega} are now complex conjugates, so
$\omega=\sqrt{2(|\zeta|+\Re\zeta)}$ is real and slightly less than
$2|x|$ for $\mu/i$ small. Then
\[ \frac{2x_0}{\omega} = x_0\sqrt{\frac{2}{|\zeta|+\Re\zeta}} , \quad
 \frac{\omega}{2\mu} = i\,\sqrt{\frac{|\zeta|+\Re\zeta}{2(e^{2\nu}-1)}},
\]
and one can verify from this that $|2x_0/\omega|\leq1$ and that
$\Im(\omega/2\mu)$ takes values in $[0,\infty)$.

Consequently, for an oblate spheroid we replace the coordinate $v$ with
the value $\arcsinh \hspace{0.01cm} \cosh v$ in order to retain  formula
\eqref{eq:prolatecoords} via the relations $2x_0/\omega=\,\cos u$ and
$\omega/(2\mu)=\,i\sinh v$, with $u\,\in\,[0,\pi]$ and
$v\in[0,\mbox{arccoth}\,e^\nu]$, and use
\eqref{eq:prolatepreharmonics} again to define the oblate harmonics.
This completes the construction of the spheroidal harmonics.

Note that $\harorig{n,0}{-}[\mu]$ vanishes identically, as do all
$\harorig{n,m}{\pm}[\mu]$ for $m\ge n+1$. Therefore when we refer to
the set $\{\harorig{n,m}{\pm}[\mu]\}$ we always exclude the indices
which refer to these trivial cases, even when we do not state
explicitly $0\le m\le n$ for the ``$+$'' case and $1\le m \le n$ for
the ``$-$'' case.  In general the polynomials
$\harorig{n,m}{\pm}[\mu]$ are not homogeneous, unlike the classical spherical harmonics \cite{Muller}
\begin{equation} \label{eq:sphericalharmonics}
  \harorig{n,m}{\pm}[0] = |x|^n P_n^m\left(\frac{x_0}{|x|}\right)
     \Phi^{\pm}_m(\phi)
 \end{equation}
where $x_0=r\cos\theta$, $x_1=r\sin\theta\cos\phi$,
$x_2=r\sin\theta\sin\phi$.

The family $\{\harorig{n,m}{\pm}[\mu]\}$ turns out to be orthogonal
with respect to the Dirichlet inner product \cite{Garabedian}, but
not in $L_2(\Omega_\mu)$. Define
\begin{equation}  \label{eq:prolateharmonics}
  \harbase{n,m}{\pm}[\mu](x) =
  \frac{\partial}{\partial x_0} \harorig{n+1,m}{\pm}[\mu](x) .
\end{equation}
Since the functions \eqref{eq:prolateharmonics}, except for the
constant factors $\coefi{n,m}$ and the rescaling of the $x$
variable, are the functions defined in \cite{Garabedian}, the main
result of that paper can be restated as follows.

\begin{theo} \label{propgarabedian} The functions
  $\harbase{n,m}{\pm}[\mu]$ ($n\geq0$) are harmonic polynomials in
  $x_0,x_1,x_2$ of degree $n$. They form a complete orthogonal family
  in the closed subspace $L_2(\Omega_\mu)\cap \Har(\Omega_\mu)$ of
  $L_2(\Omega_\mu)$. Furthermore,
\begin{equation}
    \|\harbase{n,m}{\pm}[\mu]\|_2^2 =
   (1+\delta_{0,m}) \mu^{2n+3} \coefii{n,m} 
  \int_{1}^{\frac{1}{\mu}}P_{n}^{m}(t)P_{n+2}^{m}(t)\,dt,\label{e:harqnorms}
\end{equation}
where $\delta_{m,m'}$ is the Kronecker symbol, and
\begin{equation}\label{coeficientnorm}
 \coefii{n,m}  =  \frac{\pi \, 2^{n+1} (n+m+1)(n-m+2)! (n+m+1)!}
     {(2n+1)!! (2n+3)!!}.
\end{equation}
\end{theo}

The use of the particular coefficient $\coefi{n,m}$ in
\eqref{eq:prolatepreharmonics} is for the following.
\begin{prop}
  For every $x\in\R^3$, the limit
  $\lim_{\mu\to0}\harbase{n,m}{\pm}[\mu](x)$ exists and is given by
  $ \harbase{n,m}{\pm}[0](x)=(\partial/\partial x_0)
  \harorig{n+1,m}{\pm}[0](x)$,
  where $ \harorig{n+1,m}{\pm}[0]$ is the classical spherical
  harmonic (cf.\  \eqref{eq:sphericalharmonics}).
\end{prop}

\begin{proof} It is sufficient to prove that
  $\harorig{n,m}{\pm}[\mu]\to\harorig{n,m}{\pm}[0]$.  Since $\phi$
  in \eqref{eq:prolatecoords} and \eqref{eq:sphericalharmonics} does
  not depend on $x_0$, we examine the factors
  $P_{n}^{m}(2x_0/\omega)P_{n}^{m}(\omega/(2\mu))$ in
  \eqref{eq:prolatepreharmonics}, with $\omega$ again given by
  \eqref{eq:omega}. Since
\[\sqrt{(x_0\pm\mu)^2 + x_1^2 + x_2^2} = |x|\pm(x_0/|x|)\mu +
  O(\mu^2),\]
we have $\omega=2|x|+O(\mu^2)$ as $\mu\to0$.

A direct computation using \eqref{eq:omega} shows that
$2x_0/\omega =x_0/|x|+O(\mu)$, so
$P_n^m(2x_0/\omega)\to P_n^m(x_0/x)$ as $\mu\to0$. It can be shown inductively that
$ \coefi{n,m} = 2^{-n}n!(n+m)!\sum_{k=m}^{n}\lambda_{k}^{n,m}$, where
$\lambda_{k}^{n,m} = ((n+m-k)!(n-k)!  (k-m)!k!)^{-1}$.
From the explicit representation
\[
P_{n}^{m}(t)=  \frac{n!(m+n)!}{2^n}(t^2-1)^{m/2}\displaystyle\sum_{k=m}^n\lambda_k^{n,m}(t-1)^{n-k}(t+1)^{k-m}
\]
valid for real $|t|>1$,
 we have  the required asymptotic behavior
\[   P_n^m(t) \simeq \dfrac{1}{\coefi{n,m}}   t^n  .
\]
 as $t=\omega/2\mu$ tends to infinity, which corresponds to
$\mu\to0$ for fixed $x$.
\end{proof}

The spherical harmonics are embedded in this 1-parameter
family of spheroidal harmonics.  In contrast, in treatments such as
\cite{Garabedian,Hobson,Morais2,Morais5,MoPeK,MoNgK,Nguyen}, the spheroidal harmonics degenerate
to a segment as the eccentricity of the spheroid decreases.

We introduce the notation
$\harbase{n,m}{\pm}[\mu]=\harbaseq{n,m}[\mu]\Phi_m^\pm$,
$\harorig{n,m}{\pm}[\mu]=\harorigq{n,m}[\mu]\Phi_m^\pm$ for use when the
factors $\Phi_m^\pm$ are not of interest. The following will be key in
the proof of Theorem \ref{th:monogbase} and it is based on the results of \cite{Morais2}.

\begin{prop}\label{prop:harrecurrenceformula} For each $n\geq 2$, the functions $\harbaseq{n,m}[\mu]$
  satisfy the recurrence relation
  \[ \harbaseq{n,m}[\mu] =  (n+m+1) \harorigq{n,m}[\mu] +
    \frac{\mu^2(n+m+1)(n+m)}{(2n+1)(2n-1)}\harbaseq{n-2,m}[\mu].
  \]
\end{prop}

\begin{proof}
 We will assume that $\nu<0$, because the case $\nu>0$ is similar.
  From differentiating \eqref{eq:prolatecoords},
\[ \frac{\partial}{\partial x_0} = \frac{1}{\mu(\cos^2u - \cosh^2v)}
  (\sin u \cosh v\, \frac{\partial}{\partial u}
    - \cos u \sinh v\, \frac{\partial}{\partial v}),
\]
from which the definition \eqref{eq:prolateharmonics} gives
\begin{align}\label{eq:Vn_1,m}
  \frac{(\cos^2 u - \cosh^2 v)}{\coefi{n+1,m}\mu^{n}}
  \harbaseq{n,m}
  &=   (-\sin^2 u \cosh v\, P_{n+1}^{m} (\cosh v)(P_{n+1}^{m})' (\cos u)
  \nonumber\\
 & \quad\ -  \cos u \sinh^2 v\, P_{n+1}^{m}(\cos u)(P_{n+1}^{m})' (\cosh v)).
\end{align}
There are many well-known recurrence relations for the associated
Legendre functions (see for example \cite[Ch. III]{Hobson}). The relation
\begin{equation} (1 - t^2) (P_{n+1}^m)' (t) = (n + m + 1) P_{n}^m (t)
- (n + 1) t P_{n+1}^m (t) \label{eq:legendre3}
\end{equation}
yields that \eqref{eq:Vn_1,m} is equal to $(n+m+1)$ times
\begin{align}
    \cosh v\, P_{n}^{m}(\cos u)P_{n+1}^{m}(\cosh v)
  - \cos u\, P_{n+1}^{m}(\cos u) P_{n}^{m}(\cosh v) .
\end{align}
The further relation
\begin{equation} \label{eq:legendre4}
(n - m + 1) P_{n+1}^m (t) = (2n + 1) t P_{n}^m (t) - (n + m) P_{n-1}^m (t)
\end{equation}
shows that
\begin{align*}
\harbaseq{n,m}  = & (n+m+1) \harorigq{n,m}\\
  & +   \frac{\coefi{n,m} \mu^{n}(n+m+1)(n+m)}{(\cosh^2v - \cos^2u)(2n+1)}
     [\cos u \, P_{n-1}^{m}(\cos u)P_{n}^{m}(\cosh v)\\
  & -  \cosh v \,P_{n}^{m}(\cos u) P_{n-1}^{m}(\cosh v)],
\end{align*}
as is seen after substituting
\[ \coefi{n,m} = \left(\frac{2n+1}{n-m+1}\right)\coefi{n+1,m}.
\]
Using \eqref{eq:legendre4} again, straightforward computations show that
\begin{align*}
  \harbaseq{n,m}  = &
        (n+m+1)\harorigq{n,m}\\
 &+
\frac{\coefi{n-1,m} \mu^{n}(n+m+1)(n+m)(n+m-1)}{(\cosh^2v - \cos^2u) (2n-1)(2n+1)}\times \\
& \qquad   [\cosh v \,P_{n-2}^{m}(\cos u) P_{n-1}^{m}(\cosh v) \\
& \qquad -  \cos u \,P_{n-1}^{m}(\cos u)P_{n-2}^{m}(\cosh v)].
\end{align*}
The result now follows.
\end{proof}


\section{Spheroidal monogenic  functions\label{sec:monogenics}}

Regard $\R^3$ as the subset of the quaternions
$\H=\{x_0+x_1e_1+x_2e_2+x_3e_3\}$ for which $x_3=0$.  Although this
subspace is not closed under quaternionic multiplication (which is
defined, as is usual, so that $e_i^2=-1$ and $e_ie_j=-e_je_i$ for
$i\not=j$), it is possible to carry out a great deal of the analysis analogous to that of complex numbers \cite{Delanghe2007,Gurlebeck2,Morais,Morais3,Morais4,MoraisAG}.

Consider the Cauchy-Riemann
(or Fueter) operators
\begin{align}  \label{eq:defD}
  \partial = \frac{\partial}{\partial x_0} -
       \sum_{i=1}^2\frac{\partial}{\partial x_i} , \quad
  \Dbar = \frac{\partial}{\partial x_0} +
       \sum_{i=1}^2\frac{\partial}{\partial x_i} .
  \end{align}
A smooth function $f$ defined in an open set of $\R^3$ is
(left-)monogenic when $\partial f=0$, and (left-)antimonogenic
when $\Dbar f=0$ identically. The term ``hyperholomorphic''
is also commonly used.


\subsection{Construction of orthogonal basis of monogenics}

A basis of polynomials spanning the square-integrable solutions of
$\overline{\partial} f=0$ was given in \cite{Morais2} and another in \cite{Morais5} for
prolate spheroids, via explicit formulas. However,   following
\cite{Garabedian} here we take another approach to monogenic functions,
more suitable to our purposes. We consider simultaneously the prolate
and oblate cases of spheroids. Define the \textit{basic monogenic
  spheroidal polynomials} to be
\begin{equation}
  \monog{n,m}{\pm}[\mu] = \partial \harorig{n+1,m}{\pm}[\mu].
\end{equation}
They are indeed monogenic since $\harorig{n+1,m}{\pm}[\mu]$ is
harmonic, in view of the factorization $\Delta=\Dbar\partial$ of the
Laplacian. Now we will work out explicit expressions in terms of the
orthogonal basis of harmonic functions; some examples in low degree
are exhibited in Tables \ref{tab:firstmonogenics1} and
\ref{tab:firstmonogenics2}.

As was shown in \cite{Morais2}, the equality
\begin{equation} \label{e:relationharsubinneg}
\harbaseq{n,-1} = -\frac{1}{(n+1)(n+2)} \harbaseq{n,1}
\end{equation}
can be easily verified for $n\geq 0$. These functions will appear in
the representation \eqref{eq:sphmonogformula} for the case of
zero-order monogenic polynomials (see Theorem \ref{th:monogenics} below).

\begin{theo} \label{th:monogenics}
  For each $n\geq0$ and $0\leq m \leq n+1$, the basic spheroidal
  monogenic polynomial is equal to
 \begin{align}
 \monog{n,m}{\pm}[\mu]  &=
 \harbase{n,m}{\pm}[\mu]   \nonumber \\
 & \qquad + \frac{e_1}{2}\big((n+m+1) \harbase{n,m-1}{\pm}[\mu] -
  \frac{1}{n+m+2} \harbase{n,m+1}{\pm}[\mu] \big) \nonumber\\
 & \qquad \mp \frac{e_2}{2} \big((n+m+1) \harbase{n,m-1}{\mp}[\mu] +
 \frac{1}{n+m+2} \harbase{n,m+1}{\mp}[\mu] \big) \label{eq:sphmonogformula}
 \end{align}
where the harmonic polynomials $\harbase{n,m}{\pm}[\mu]$ were defined
in \eqref{eq:prolateharmonics}. The $\monog{n,m}{\pm}[\mu]$ are
polynomials in $\mu^2$ as well as in $x_0,x_1,x_2$.
\end{theo}

\begin{proof}
  The full operator \eqref{eq:defD} in spheroidal coordinates \eqref{eq:prolatecoords} is
\begin{align*}
  \partial =& \frac{1}{\mu(\cosh^2v-\cos^2u)}
       \bigg[(\cos u\sinh v  \frac{\partial}{\partial v} -
      \sin u \cosh v   \frac{\partial}{\partial u}) \\
     &\qquad - (e_1\cos\phi+e_2\sin \phi)
       (\cos u\sinh v \frac{\partial}{\partial u} +
        \sin u\cosh v \frac{\partial}{\partial v}) \\
     &\qquad\qquad\qquad  \times\frac{(-e_1\sin\phi+e_2\cos\phi)}
       {\mu \sin u\sinh v} \frac{\partial}{\partial\phi} \bigg].
\end{align*}
The first line of this expression applied to
$\harorig{n+1,m}{\pm}[\mu]$ produces the scalar part of
$\monog{n,m}{\pm}[\mu]$ in
\eqref{eq:sphmonogformula} and was calculated in \cite{Morais2,Morais5}. For
the nonscalar part, we use the relation \eqref{eq:legendre3} to obtain
\begin{align*}
 \quad & \frac{2}{\mu^{n+1} \coefi{n+1,m}\Phi^\pm_m}
    ( \cos u\sinh v \,\partial_u + \sin u\cosh v \,\partial_v)
    \harorig{n+1,m}{\pm}[\mu] \hspace*{.22\textwidth} \\
 &\quad\quad = (n+m+1)(n-m+2)\big(\sin u \cosh v P_{n+1}^m(\cos u)P_{n+1}^{m-1}(\cosh v)\\
 &\quad\qquad\quad\quad\  - \cos u\sinh v\, P_{n+1}^{m-1}(\cos u) P_{n+1}^m (\cosh v)\big)\\
 &\quad\qquad + \big(\sin u \cosh v P_{n+1}^m(\cos u)P_{n+1}^{m+1}(\cosh v)\\
 &\quad \quad\qquad\quad\ + \cos u\sinh v P_{n+1}^{m+1}(\cos u) P_{n+1}^m(\cosh v) \big) .
\end{align*}
Next, the relation
\[ \sqrt{1-t^2}P_{n+1}^m(t)=(n-m)tP_{n+1}^{m-1}(t)-(n+m)P_n^{m-1}(t)
\]
(valid for $|t| < 1$, and replacing $1-t^2$ with $t^2-1$ for $|t|>1$) produces
\begin{align}
  -\frac{(\cosh^2v-\cos^2u)}{\mu^n\coefi{n+1,m-1}}
  \harbaseq{n,m-1}
  & = \sin u\cosh v   \,P_{n+1}^m(\cos u)P_{n+1}^{m-1}(\cosh v) \nonumber\\
  &\quad - \cos u\sinh v  \, P_{n+1}^m(\cosh v)P_{n+1}^{m-1}(\cos u).
 \end{align}
Furthermore, using the expression
\[ (1-t^2)^{1/2}\,P_{n+1}^m(t)=\frac{1}{2n+3}\,(P_{n+2}^{m+1}(t)-P_{n}^{m+1}(t)),
\]
and its counterpart for $|t|>1$, and then applying
\eqref{eq:legendre4}, we arrive at
\begin{align}
\cosh v\sin u P_{n+1}^m(\cos u)P_{n+1}^{m+1}(\cosh v)&+ \sinh v\cos u P_{n+1}^m(\cosh v)P_{n+1}^{m+1}(\cos u) \nonumber\\
& =  \frac{(\cosh^2v-\cos^2u)\coefi{n+1,m+1}}{(n-m+1)\mu^n} \harbaseq{n,m+1}.
\end{align}
Similarly, one can prove that
\begin{align}
\frac{1}{\sin u \sinh v}\partial_{\phi}\harorig{n+1,m}{\pm}[\mu]
 = -\frac{\mu^{n+1}\Phi_m^\mp}{2\coefi{n+1,m} (\cosh^2v-\cos^2u)}&\nonumber\\
\times\bigg[(n+m)(n+m+1)(n-m+2) \harbaseq{n,m-1}  & +  \frac{1}{n-m+1} \harbaseq{n,m+1}\bigg].
\end{align}
Combining these three formulas one easily obtains the desired expressions for
$(\partial/\partial x_1)\harorig{n+1,m}{\pm}[\mu]$ and
$(\partial/\partial x_2)\harorig{n+1,m}{\pm}[\mu]$.

Since the basic spheroidal harmonics of \cite{Garabedian} are
polynomials of degree $n$, it is clear that the operations of
rescaling by $1/\mu$ or $i/\mu$ and multiplying by $\mu^n$ implied in
\eqref{eq:prolatepreharmonics} assure that $\harorig{n,m}{\pm}[\mu]$
are polynomials in $\mu$.  From the discussion of Section
\ref{sec:harmonics} it is clear that $-\mu$ produces the same
results as $\mu$, so the only powers of $\mu$ which appear are even.
\end{proof}

\begin{table}[hb!]
\[ \begin{array}{|c|c|l|}
\hline
n & m &  \monog{n,m}{\pm}\\
\hline
\multirow{3}{*}{0} & 0 & \s\monog{0,0}{+}=1\\[.5ex]
\cline{2-3}
& \multirow{2}{*}{1} & \s\monog{0,1}{+}=e_1\\
& &\s\monog{0,1}{-}=e_2\\[.5ex]
\cline{1-3}
\multirow{6}{*}{1}& 0 & \s\monog{1,0}{+}=2x_0+x_1e_1+x_2e_2 \\
\cline{2-3}
 & \multirow{2}{*}{1} & \s\monog{1,1}{+}=-3x_1+3x_0e_1\\
 & & \s\monog{1,1}{-}=-3x_2+3x_0e_2\\[.5ex]
 \cline{2-3}
 & \multirow{2}{*}{2} & \s\monog{1,2}{+}=-6x_1e_1+6x_2e_2 \\
 & & \monog{1,2}{-}=-6x_2e_1-6x_1e_2\\[.5ex]
 \cline{1-3}
 \multirow{7}{*}{2}& 0 & \s\monog{2,0}{+}=\big(3x_0^2-\dfrac{3x_1^2}{2}-\dfrac{3x_2^2}{2}-\dfrac{3\mu^2}{5}\big)+3x_0x_1e_1+3x_0x_2e_2 \\[1ex]
 \cline{2-3}
 & \multirow{2}{*}{1} & \s\monog{2,1}{+}=-12x_0x_1+\big(6x_0^2-\dfrac{9x_1^2}{2}-\dfrac{3x_2^2}{2}-\dfrac{6\mu^2}{5}\big)e_1-3x_1x_2e_2 \\
 & & \s\monog{2,1}{-}=-12x_0x_2-3x_1x_2e_1+\big(6x_0^2-\dfrac{3x_1^2}{2}-\dfrac{9x_2^2}{2}-\dfrac{6\mu^2}{5}\big)e_2 \\[1ex]
 \cline{2-3}
 & \multirow{2}{*}{2} & \s\monog{2,2}{+}=15x_1^2-15x_2^2-30x_0x_1e_1+30x_0x_2e_2  \\
 & & \s\monog{2,2}{-}=30x_1x_2-30x_0x_2e_1-30x_0x_1e_2\\[.5ex]
 \cline{2-3}
 & \multirow{2}{*}{3} & \s\monog{2,3}{+}=\big(45x_1^2-45x_2^2\big)e_1-90x_1x_2e_2\\
 & & \s\monog{2,3}{-}=90x_1x_2e_1+\big(45x_1^2-45x_2^2\big)e_2\\[.5ex]
 \cline{1-3}
\end{array} \]
\caption{Spheroidal monogenic basis polynomials of degree $n=0,1,2$. Observe
  that the parameter $\mu$ appears when $|n-m|\geq2$. For each $n$, the last two
  entries are monogenic constants (and the first entry for $n=0$). }
  \label{tab:firstmonogenics1}
\end{table}

\begin{table}[b!]
\[ \begin{array}{|c|c|l|}\hline
n & m & \monog{n,m}{\pm}\\
\hline
\multirow{16}{*}{3} & &
 {\s\monog{3,0}{+}=\big(4x_0^3-6x_0x_1^2-6x_0x_2^2-\dfrac{12x_0\mu^2}{7}\big)}\\
 & 0 &\qquad\quad +\big(6x_0^2x_1-\dfrac{3 x_1^3}{2}-\dfrac{3x_1x_2^2}{2}-\dfrac{6x_1\mu^2}{7}\big)e_1\\
 & &\qquad\quad   +\big(6x_0^2x_2-\dfrac{3 x_1^2x_2}{2}-\dfrac{3x_2^3}{2}-\dfrac{6x_2\mu^2}{7}\big)e_2\\[1ex]
 \cline{2-3}
 & & \s\monog{3,1}{+}=-30x_0^2x_1+\dfrac{15x_1^3}{2}+\dfrac{15x_1x_2^2}{2}+\dfrac{30x_1\mu^2}{7}\\
 & 1 &\qquad\quad  +\big(10x_0^3-\dfrac{45x_0x_1^2}{2}- \dfrac{15x_0x_2^2}{2}-\dfrac{30x_0\mu^2}{7}\big)e_1-15x_0x_1x_2e_2\\[.5ex]
 & & \s\monog{3,1}{-}=-30x_0^2x_2+\dfrac{15x_1^2x_2}{2}+\dfrac{15x_2^3}{2}+\dfrac{30x_2\mu^2}{7}\\
 & &\qquad\quad  -15x_0x_1x_2e_1+\big(5x_0^3-\dfrac{15x_0x_1^2}{2}- \dfrac{45x_0x_2^2}{2}-\dfrac{30x_0\mu^2}{7}\big)e_2\\[1.2ex]
 \cline{2-3}
 & &  \s\monog{3,2}{+}=90x_0x_1^2-90x_0x_2^2+\big(-90x_0^2x_1+30x_1^3+\dfrac{90x_1\mu^2}{7}\big)e_1\\
 & &\qquad\quad  +\big(90x_0^2x_2-30x_2^3-\dfrac{90x_2\mu^2}{7}\big)e_2\\
 & 2 & \s\monog{3,2}{-}=180x_0x_1x_2+\big(-90x_0^2x_2+45x_1^2x_2+15x_2^3+\dfrac{90x_2\mu^2}{7}\big)e_1\\
 & &\qquad\quad  +\big(-90x_0^2x_1+15x_1^3+45x_1x_2^2+\dfrac{90x_1\mu^2}{7}\big)e_2\\[1ex]
 \cline{2-3}
 & 3 & \s\monog{3,3}{+}=-105x_1^3+315x_1x_2^2+\big(315x_0x_1^2-315x_0x_2^2\big)e_1-630x_0x_1x_2e_2\\
 & & \s\monog{3,3}{-}=-315x_1^2x_2+105x_2^3+630x_0x_1x_2e_1+\big(315x_0x_1^2-315x_0x_2^2\big)e_2\\[.5ex]
 \cline{2-3}
 & 4 & \s\monog{3,4}{+}=(-420x_1^3+1260x_1x_2^2)e_1+(1260x_1^2x_2-420x_2^3)e_2\\
 & & \s\monog{3,4}{-}=(-1260x_1^2x_2+420x_2^3)e_1+(-420x_1^3+1260x_1x_2^2)e_2\\[.5ex]
 \hline
\end{array} \]
\caption{Spheroidal monogenic polynomials of degree $n=3$.  }
  \label{tab:firstmonogenics2}
  \end{table}

We will write
$\langle \cdot,\cdot \rangle_{[\mu]}= \langle \cdot,\cdot \rangle_{L_2
  (\Omega_\mu, \R^3)}$
for the $L_2$ inner product in the spheroidal domain $\Omega_\mu$, and
$\|\cdot\|_{[\mu]}$ for the corresponding norm.

\begin{theo}\label{th:monogbase} For fixed $\mu$,
  the monogenic polynomials $\monog{n,m}{\pm}[\mu]$ are orthogonal
  with respect to the inner product
  $\langle \cdot,\cdot \rangle_{[\mu]}$.  Their norms are given by
\begin{align*}
 \|\monog{n,m}{\pm}&\|_{[\mu]}^2
   = \dfrac{\pi \, \mu^{2n+3}}
    {(n+2)(n+m+2)(2n+1)!! (2n+3)!!} \,\big[  \\
  & \!\! (n+2)(n+m)(n+m+1)(n-m+3)! (n+m+2)! I[n,m-1] \\[1ex]
  & + 2\delta_{0,m}(n+m+2) (n+1)! (n+2)! I[n,1]\\
  & + (n+2) (n-m+1)! (n+m+2)! \big( I(n,m+1)\\
  & \qquad   + 2(n-m+2)(n+m+1)(1+\delta_{0,m})I[n,m]\big)\big]
   \end{align*}
  where
  \begin{equation} \label{eq:Imn}
     I[n,m]=\int_{1}^{1/\mu}P_{n}^{m}(t)\,P_{n+2}^{m}(t)\,dt.
  \end{equation}
 \end{theo}

\begin{proof}
  Throughout this proof, we will denote by $[f]_i$ ($i=0,1,2$)
  the components of a function $f\colon\Omega_{\mu} \to \R^3$. Thus
\begin{align*}
\langle \monog{n_1,m_1}{\pm}[\mu], \monog{n_2,m_2}{\pm}[\mu] \rangle_{[\mu]}
=&  \int_{\Omega_\mu} \big(\left[\monog{n_1,m_1}{\pm}[\mu]\right]_0
\left[\monog{n_2,m_2}{\pm}[\mu]\right]_0\\
&\qquad + \left[\monog{n_1,m_1}{\pm}[\mu]\right]_1
\left[\monog{n_2,m_2}{\pm}[\mu]\right]_1\\
&\qquad +  \left[\monog{n_1,m_1}{\pm}[\mu]\right]_2
\left[\monog{n_2,m_2}{\pm}[\mu]\right]_2 \big)  \,dx
\end{align*}
where $dx = dx_0\,dx_1\,dx_2$.
By Theorem \ref{th:monogenics} and Proposition \ref{propgarabedian},
\begin{equation}\label{eq:scmonogenicnorm}
\int_{\Omega_\mu} \left[\monog{n_1,m_1}{\pm}[\mu]\right]_0
\left[\monog{n_2,m_2}{\pm}[\mu]\right]_0  \,dx =
  \| \harbase{n_1,m_1}{\pm}[\mu] \|^2_{[\mu]} \,  \delta_{n_1,n_2}
  \delta_{m_1,m_2}.
\end{equation}
Thus, to verify the orthogonality of the $\monog{n,m}{\pm}[\mu]$ it
suffices to show that the vector parts of the functions
$\monog{n,m}{\pm}[\mu]$ are orthogonal.

Expanding the integrands and applying the trigonometric identities
\begin{align*}
  \Phi_{m_1-1}^\pm \Phi_{m_2-1}^\pm   + \Phi_{m_1-1}^\mp \Phi_{m_2-1}^\mp
  &= \Phi_{m_1-m_2}^+ ,\\
  \Phi_{m_1+1}^\pm \Phi_{m_2+1}^\pm   + \Phi_{m_1+1}^\mp \Phi_{m_2+1}^\mp
  &= \Phi_{m_1-m_2}^+ ,\\
  -\Phi_{m_1-1}^\pm \Phi_{m_2+1}^\pm  + \Phi_{m_1-1}^\mp \Phi_{m_2+1}^\mp
  &= \mp \Phi_{m_1+m_2}^+ ,\\
  -\Phi_{m_1+1}^\pm \Phi_{m_2-1}^\pm  + \Phi_{m_1+1}^\mp \Phi_{m_2-1}^\mp
  &= \mp \Phi_{m_1+m_2}^+,
\end{align*}
we obtain that
\begin{align*}
   \int_{\Omega_\mu} &\big(\left[\monog{n_1,m_1}{\pm}[\mu]\right]_1
\left[\monog{n_2,m_2}{\pm}[\mu]\right]_1
  + \left[\monog{n_1,m_1}{\pm}[\mu]\right]_2
  \left[\monog{n_2,m_2}{\pm}[\mu]\right]_2\big) \,dx \\
 =& \frac{1}{4} \bigg(p_1p_2
   \int_{\Omega_\mu}\harbaseq{n_1,m_1-1}[\mu] \harbaseq{n_2,m_2-1}[\mu]\Phi_{m_1-m_2}^+
    \, dx   \\
 &  \mp \frac{p_1}{p_2+1} \int_{\Omega_\mu} \harbaseq{n_1,m_1-1}[\mu]
\harbaseq{n_2,m_2+1}[\mu] \Phi_{m_1+m_2}^+  \, dx \\
   & \mp \frac{p_2}{p_1+1} \int_{\Omega_\mu} \harbaseq{n_1,m_1+1}[\mu]
\harbaseq{n_2,m_2-1}[\mu]\Phi_{m_1+m_2}^+  \, dx \\
   & + \frac{1}{(p_1+1) (p_2+1)} \int_{\Omega_\mu} \harbaseq{n_1,m_1+1}[\mu]
\harbaseq{n_2,m_2+1}[\mu] \Phi_{m_1-m_2}^+  \, dx    \bigg )
\end{align*}
where $p_i=m_i+n_i+1$ $(i=1,2)$. We continue the calculation only for the
prolate case, applying the coordinates \eqref{eq:prolatecoords}
which give
\[ dx = dR\,d\phi, \]
where $dR=\mu^3(\cosh^2 v - \cos^2 u) \sin u \sinh v\,du\,dv$.

The identities
$\int_{0}^{2\pi}\Phi_{m_1\pm m_2}(\phi) d\phi=2\pi\delta_{m_1,m_2}$
for $m_1,m_2>0$ imply that
\begin{align*}
\int_{\Omega_\mu} & \left(\left[\monog{n_1,m_1}{\pm}[\mu]\right]_1
\left[\monog{n_2,m_2}{\pm}[\mu]\right]_1 +
\left[\monog{n_1,m_1}{\pm}[\mu]\right]_2
\left[\monog{n_2,m_2}{\pm}[\mu]\right]_2\right) \,  dx   \\
 =& \frac{\pi p_1(n_2+m_1+1)\delta_{m_1,m_2}}{2} \int_{0}^{1/\mu}
\int_{0}^{\pi}\harbaseq{n_1,m_1-1}[\mu] \harbaseq{n_2,m_1-1}[\mu]\, dR \\
 & \mp \pi \frac{n_2+1}{2(n_1+2)}\delta_{m_1,0} \int_{0}^{1/\mu} \int_{0}^{\pi}
\harbaseq{n_1,1}[\mu] \harbaseq{n_2,-1}[\mu]\,dR \\
  & \mp \pi \frac{n_1+1}{2(n_2+2)}\delta_{m_1,0} \int_{0}^{1/\mu} \int_{0}^{\pi}
\harbaseq{n_1,-1}[\mu] \harbaseq{n_2,1}[\mu] \,dR \\
 & + \frac{\pi}{2(p_1+1)(n_2+m_1+1)}\delta_{m_1,m_2} \int_{0}^{1/\mu}
\int_{0}^{\pi} \harbaseq{n_1,m_1+1}[\mu] \harbaseq{n_2,m_1+1}[\mu] \,dR.
\end{align*}
In consequence, using \eqref{e:relationharsubinneg}, we have
\begin{align*}
  \int_{\Omega_\mu}& \left(\left[\monog{n_1,m_1}{\pm}[\mu]\right]_1
  \left[\monog{n_2,m_2}{\pm}[\mu]\right]_1
 + \left[\monog{n_1,m_1}{\pm}[\mu]\right]_2
  \left[\monog{n_2,m_2}{\pm}[\mu]\right]_2\right)   \,dx \\
  =& \frac{\pi p_1(n_2+m_1+1)\delta_{m_1,m_2}}{2}
\int_{0}^{1/\mu} \int_{0}^{\pi}\harbaseq{n_1,m_1-1}[\mu] \harbaseq{n_2,m_1-1}[\mu]
\,dR \\
 & \pm \frac{\pi}{(n_1+2)(n_2+2)}\delta_{m_1,0} \int_{0}^{1/\mu}
\int_{0}^{\pi} \harbaseq{n_1,1}[\mu] \harbaseq{n_2,1}[\mu] \,dR \\
  &  + \frac{\pi}{2p_1(n_2+m_1+1)}\delta_{m_1,m_2} \int_{0}^{1/\mu}
 \int_{0}^{\pi} \harbaseq{n_1,m_1+1}[\mu]\harbaseq{n_2,m_1+1}[\mu] \,dR.
\end{align*}
Using Proposition \ref{prop:harrecurrenceformula}, and applying again the
orthogonality of Theorem \ref{propgarabedian}, we are left with
\begin{align}
  \int_{\Omega_\mu}& \left(\left[\monog{n_1,m_1}{\pm}[\mu]\right]_1
 \left[\monog{n_2,m_2}{\pm}[\mu]\right]_1 +
 \left[\monog{n_1,m_1}{\pm}[\mu]\right]_2
 \left[\monog{n_2,m_2}{\pm}[\mu]\right]_2\right) \, dx  \nonumber\\
  &=   \frac{\pi \mu^{2n_1+3}}{(n_1+2) (2n_1+1)!! (2n_1+3)!!} \big[(n_1+2)(n_1+m_1+1)! \nonumber\\
  &\quad \times \big((n_1+m_1)(n_1+m_1+1)(n_1-m_1+3)! I[n_1,m_1-1]  \nonumber\\[1ex]
  &\qquad + (n_1-m_1+1)! I[n_1,m_1+1]\big)\nonumber\\
  &\qquad +2 (n_1+1)! (n_1+2)! I[n_1,1]\delta_{0,m}\big] \delta_{m_1,m_2}\delta_{n_1,n_2}
   \label{eq:vecmonogenicnorm}
\end{align}
with $I[n,m]$ defined in \eqref{eq:Imn}.  Combining
  \eqref{eq:scmonogenicnorm} and \eqref{eq:vecmonogenicnorm}, we
  conclude that
\begin{align*}
  \langle\monog{n_1,m_1}{+},\monog{n_2,m_2}{+}\rangle_{[\mu]} &=0
   \text{ when } n_1\neq n_2 \text{ or } m_1\neq m_2,\\
  \langle\monog{n_1,m_1}{-},\monog{n_2,m_2}{-}\rangle_{[\mu]} &=0
    \text{ when } n_1\neq n_2\text{ or } m_1\neq m_2.
\end{align*}
Using once more the orthogonality of the system $\{\Phi_m^\pm\}$
on $[0,2\pi]$, we conclude that
\[ \langle\monog{n_1,m_1}{\pm},\monog{n_2,m_2}{\mp}\rangle_{[\mu]} = 0
\]
when the indices do not coincide.  The calculation of the norms comes
from taking $n_1=n_2$ and $m_1=m_2$ in \eqref{eq:vecmonogenicnorm} and
adding the expression \eqref{e:harqnorms}.
\end{proof}

In the next subsection we make precise how the monogenic polynomials sit
in the space of harmonic polynomials.  It is well known (cf.\
\cite{Leutwiler}) that the dimension of the space
$\M^{(n)}$ of homogeneous monogenic polynomials of degree $n$ in
$(x_0,x_1,x_2)$ is $2n+3$ (this does not depend on the domain
$\Omega$).  Since the polynomials we are working with are not
homogeneous, we consider the space $\M_*^{(n)}=\bigcup_{0\le k\le n}\M^{(n)}$ of monogenic
polynomials of degree $n$, a class which is not altered by adding
monogenic polynomials of lower degree. Thus
\begin{equation} \label{eq:dimM*}
  \dim  \M_*^{(n)}  = \sum_{k=0}^{n} (2k+3) = (n+3)(n+1).
\end{equation}

 Consider the collections of $2n+3$ polynomials
\[  B_n[\mu]=  \{\monog{k,m}{+}[\mu],\ 0 \leq m \leq k+1\} \cup
  \{ \monog{k,m}{-}[\mu], 1 \leq m \leq k+1\}. \]
By Theorem \ref{th:monogbase} and \eqref{eq:dimM*}, the union
\[ \bigcup_{0 \leq k \leq n} B_n[\mu] \]
is an orthogonal basis for $\M_*^{(n)}$.  By the symmetric form taken
by $\monog{m,n}{\pm}[\mu]$ in \eqref{eq:sphmonogformula}, we know that
when $m \neq 0$,
\[ \|\monog{n,m}{+}[\mu] \|_{[\mu]} =  \|\monog{n,m}{-}[\mu]\|_{[\mu]}.
\]


\subsection{Spheroidal ambigenic polynomials}

Facts about antimonogenic functions are generally trivial
modifications of facts about monogenic functions, obtained by taking the conjugate. However, in order to discuss contragenic functions below
it will be necessary to discuss first the subspace of the
$\R^3$-valued harmonic functions generated by the monogenic and
antimonogenic functions together. Elements of this space were termed
\textit{ambigenic} functions in \cite{Alvarez}.

It is known \cite{Morais3}, and easy to verify, that $f$ is antimonogenic if and
only if $\overline{f}$ is monogenic. The decomposition of an ambigenic
function as a sum of a monogenic and an antimonogenic function is not
unique, so we must take into account the set
$\M(\Omega)\cap\overline{\M}(\Omega)$ of monogenic constants in the
domain $\Omega\subseteq\R^3$. Monogenic constants do not depend on
$x_0$ and can be expressed as
\[ f = a_0+f_1e_1+f_2e_2,
\]
where $a_0\in\R$ is a constant, and $f_1-if_2$ is an ordinary
holomorphic function of the complex variable $x_1+ix_2$. There are
natural projections of $\M(\Omega)$ onto the subspaces
\begin{align*}
 \Sc\M(\Omega)  &=  \{\Sc f|\ f\in\M(\Omega)\} \subseteq\Har_\R(\Omega) \\
 \Vec\M(\Omega) &=  \{\Vec f|\ f\in\M(\Omega)\}
  \subseteq\Har_{\{0\}\oplus\R^2}(\Omega),
\end{align*}
where $\Har_\R(\Omega)$ denotes the space of real-valued harmonic
functions defined in $\Omega$.  Note that
$\Sc\M(\Omega)=\Sc\overline{\M}(\Omega)$ and
$\Vec\M(\Omega)=-\Vec\overline{\M}(\Omega)$. When $\Omega$ is
simply-connected, $\Sc\M(\Omega)=\Har_\R(\Omega)$.

It is known \cite{Alvarez}, that the real dimension of the
space $\M^{(n)}+\overline{\M}^{(n)}$ of homogeneous ambigenic
polynomials is $4n+4$ when $n\ge1$.  As discussed in the previous
section, the basis polynomials for spheroidal functions are not
homogenous. The dimension of the space
$\M_*^{(n)}+\overline{\M}_*^{(n)}$ of ambigenic polynomials of degree
at most $n$ is
 \begin{align}
  \dim (\M_*^{(n)}+\overline{\M}_*^{(n)})
  =&   \sum_{k=0}^{n}
  \dim(\M^{(k)}+\overline{\M}^{(k)}) \nonumber\\
  =&  3 + \sum_{k=1}^{n}(4k+4)= 2n(n+3)+3.
\end{align}

Observe that we have by \eqref{eq:sphmonogformula} that
\[ \monog{n,n+1}{\pm}[\mu] =
  (n+1)(\harbase{n,n}{\pm}[\mu]e_1\mp\harbase{n,n}{\mp}[\mu]e_2),
\]
and
\[ \harbase{n,n}{\mp}[\mu] = (-1)^n(2n+1)!!\,
(x_1^2+x_2^2)^{n/2}.
\]
The first
equation shows that $\monog{n,n+1}{\pm}[\mu]$ has vanishing scalar
part; i.e.\ they are the negatives of their conjugates. Consequently,
the $\monog{n,n+1}{\pm}[\mu]$ are monogenic constants. This
observation makes it possible to give a basis for the ambigenic
polynomials defined in spheroidal domains. In the following we take
into account the fact that
\begin{lema} For $m\not=0$,
  \[ \langle\monog{k,m}{+}[\mu],\, \antimonog{k,m}{+}[\mu] \rangle_{[\mu]}
  =  \langle\monog{k,m}{-}[\mu],\, \antimonog{k,m}{-}[\mu] \rangle_{[\mu]}.
\]
\end{lema}
\begin{proof}
Indeed,
\begin{align*}
  \langle & \monog{k,m}{+}[\mu],\, \antimonog{k,m}{+}[\mu] \rangle_{[\mu]}
   = \int_{\Omega_\mu} \left(\left[\monog{k,m}{+}[\mu]\right]_0^2
   - \left[\monog{k,m}{+}[\mu]\right]_1^2-\left[\monog{k,m}{+}[\mu]\right]_2^2 \right) dx \\
  =& \int_{0}^{\pi}\int_{0}^{1/\mu}(\harbaseq{k,m}[\mu])^2dvdu\int_{0}^{2\pi}\cos^2m\phi\,d\phi\\
 & - \frac{1}{4}\int_{0}^{\pi}\int_{0}^{1/\mu}\big((k+m+1)\harbaseq{k,m-1}[\mu]
   - \frac{1}{k+m+2}\harbaseq{k,m+1}[\mu]\big)^2 \,dv\,du\\
 & \times   \int_{0}^{2\pi}\cos^2m\phi \,d\phi\\
 & - \frac{1}{4}\int_{0}^{\pi}\int_{0}^{1/\mu}\big( (k+m+1)\harbaseq{k,m-1}[\mu]
   - \frac{1}{k+m+2}\harbaseq{k,m+1}[\mu]\big)^2 \,dv\,du\\
 & \times  \int_{0}^{2\pi}\sin^2(m\phi)d\phi.
\end{align*}
Since $m\neq0$, the two values  $\int_{0}^{2\pi}\Phi_m^\pm(\phi)^2\,d\phi$
are equal,
and therefore
\begin{align*}
\langle\monog{k,m}{+}[\mu],\antimonog{k,m}{+}[\mu]\rangle_{[\mu]}  = &
\int_{\Omega_\mu} \left(\left[\monog{k,m}{-}[\mu]\right]_0^2-\left[\monog{k,m}{-}[\mu]\right]_1^2-\left[\monog{k,m}{-}[\mu]\right]_2^2 \right)dx
\\  = & \langle\monog{k,m}{-}[\mu],\antimonog{k,m}{-}[\mu]\rangle_{[\mu]}.
\end{align*}
\end{proof}

It is not possible to extract from the list $\{\monog{n,m}{\pm},\  \antimonog{n,m}{\pm}\}$
an orthogonal basis of ambigenic functions, but only a small modification is necessary.
  Define the functions
 \begin{align*}
  \ambi{n,m}{++}[\mu]  = & \monog{n,m}{+}[\mu], \\
  \ambi{n,m}{-+}[\mu] = & \monog{n,m}{-}[\mu], \\
  \ambi{n,m}{+-}[\mu]  = &   \antimonog{n,m}{+}[\mu]
    - \coefiii{n,m}[\mu]\monog{n,m}{+}[\mu], \\
  \ambi{n,m}{--}[\mu]  = &  \antimonog{n,m}{-}[\mu]
    - \coefiii{n,m}[\mu]\monog{n,m}{-}[\mu], ,
\end{align*}
where
\[ \coefiii{n,m}[\mu] = \left\{\begin{array}{cc}
      \frac{\langle\monog{n,m}{+}[\mu],\antimonog{n,m}{+}[\mu]\rangle_{[\mu]}}{\|\monog{n,m}{+}[\mu]\|^2_{[\mu]}},
        & \text{ if }0\leq m\leq n, \\[1ex]
      0, & \text{ if }m=n+1,  \end{array} \right.
\]
and
\begin{align*}
\ambi{0,m}{++}[\mu] = & \monog{0,m}{+}[\mu] \text{ for } m=0,1,\\
\ambi{0,1}{-+}[\mu] = & \monog{0,1}{-}[\mu].
\end{align*}

\begin{prop} \label{prop:ambibasis}
The collection of $2n(n+3)+3$ polynomials
\begin{align*}
  \{& \ambi{k,m}{++} :\ 0\le m\le k+1\} \cup
   \{   \ambi{k,m}{-+}:\  0\le m\le k\}   \\
& \cup  \{   \ambi{k,m}{+-}:\  0\le m\le k\} \cup
   \{ \ambi{k,m}{--}:\ 0\le m\le k+1\} ,
\end{align*}
$0\le k\le n$, is an orthogonal basis in $L_2(\Omega_\mu)$ for the
subspace of ambigenic polynomials of degree at most $n$.
\end{prop}

\begin{proof}
  Throughout this proof, in view of the fact that $\mu$ is fixed, we
  simply write $\monog{k,m}{\pm}$, $\ambi{k,m}{\pm\pm}$,
   $\coefiii{k,m}$ for
  $\monog{k,m}{\pm}[\mu]$, $\ambi{k,m}{\pm,\pm}[\mu]$,
   $\coefiii{k,m}[\mu]$.
    Since there are $2n(n+3)+3$ ambigenic functions in the given list, it
  suffices to prove the orthogonality to conclude that they generate the
  ambigenic polynomials. Because the set
\[ \{\monog{k,0}{+},\monog{k,m}{+},\monog{k,m}{-}|\ k=0,\dots,n,\ m=1,\dots,k+1\}
\]
is an orthogonal basis of $\M_*^{(n)}$ in $\Omega_\mu$, it follows
at once that
\[   \langle\ambi{k,m}{++},\,\overline{\ambi{k,m}{-+}} \rangle_{[\mu]}
   = \langle\ambi{k,m}{++},\,\overline{\ambi{k,m}{--}} \rangle_{[\mu]}
   = \langle\ambi{k,m}{+-},\,\overline{\ambi{k,m}{-+}} \rangle_{[\mu]}
   = \langle\ambi{k,m}{+-},\,\overline{\ambi{k,m}{--}} \rangle_{[\mu]}
   =0.
\]
Since
\begin{align*}
  \langle \ambi{k_1,m_1}{+-},\, \ambi{k_2,m_2}{+-}\rangle_{[\mu]}
 = &  \langle \antimonog{k_1,m_1}{+}-\coefiii{k_1,m_1}\monog{k_1,m_1}{+},
 \  \antimonog{k_2,m_2}{+}-\coefiii{k_2,m_2}\monog{k_2,m_2}{+}\rangle_{[\mu]}\\
 = & \langle \antimonog{k_1,m_1}{+},\ \antimonog{k_2,m_2}{+}\rangle_{[\mu]}
  - \coefiii{k_2,m_2}\langle \antimonog{k_1,m_1}{+}, \
        \monog{k_2,m_2}{+}\rangle_{[\mu]}\\
  & - \coefiii{k_1,m_1}\langle\monog{k_1,m_1}{+},\
    \antimonog{k_2,m_2}{+}\rangle_{[\mu]} \\
  & + \coefiii{k_1,m_1}\coefiii{k_2,m_2}
     \langle\monog{k_1,m_1}{+},\monog{k_2,m_2}{+}\rangle_{[\mu]},
\end{align*}
it will be enough to study
$\langle \antimonog{k_1,m_1}{+},\monog{k_2,m_2}{+}\rangle_{[\mu]}$
and
$\langle\monog{k_1,m_1}{+},\,\antimonog{k_2,m_2}{+}\rangle_{[\mu]}$:
\begin{align*}
 \langle \antimonog{k_1,m_1}{+},\monog{k_2,m_2}{+}\rangle_{[\mu]}
 = &  \int_{\Omega_\mu}\big(\left[\monog{k_1,m_1}{+}\right]_0\left[\monog{k_2,m_2}{+}\right]_0 \\
& - \big(\left[\monog{k_1,m_1}{+}\right]_1\left[\monog{k_2,m_2}{+}\right]_1
  +\left[\monog{k_1,m_1}{+}\right]_2\left[\monog{k_2,m_2}{+}\right]_2 )\big) \,dx,
\end{align*}
but from the proof of Proposition \ref{th:monogbase}, we obtain that
\[   \langle \antimonog{k_1,m_1}{+},\monog{k_2,m_2}{+}\rangle_{[\mu]}
  =  (\|\Sc\monog{k_1,m_1}{+}\|^2_{[\mu]} -
    \|\Vec\monog{k_1,m_1}{+}\|^2_{[\mu]} )\delta_{k_1,k_2}\delta_{m_1,m_2}.
\]
Now we note that
\[ \langle\ambi{k_1,m_1}{++},\ambi{k_2,m_2}{+-}\rangle_{[\mu]}
  =\langle\monog{k_1,m_1}{+},\antimonog{k_2,m_2}{+}-
  \coefiii{k_2,m_2}\monog{k_2,m_2}{+}\rangle_{[\mu]}.
\]
By the above observations, these functions are orthogonal when
$k_1\neq k_2$ or $m_1\neq m_2$, and when the indices coincide,
\[ \langle\ambi{k,m}{++},\ambi{k,m}{+-}\rangle_{[\mu]}=
 \langle\monog{k,m}{+},\antimonog{k,m}{+}\rangle_{[\mu]}-
  \frac{\langle\monog{k,m}{+},\
 \antimonog{k,m}{+}\rangle_{[\mu]}}{\|\monog{k,m}{+}\|^2_{[\mu]}}
 \langle\monog{k,m}{+},\monog{k,m}{+}\rangle_{[\mu]}=0.
\]
Moreover, by the orthogonality of the system
$\{\Phi_k^+,\Phi_l^-|\ k\geq0,\ l>0\}$, it is clear that
$\langle\ambi{k_1,m_1}{++},\ambi{k_2,m_2}{--}\rangle_{[\mu]}=0$,
and further
 $\langle\ambi{k,m}{++},\ambi{k,m}{--}\rangle_{[\mu]}=0$. Finally,
\[  \langle\ambi{k,m}{-+},\ambi{k,m}{--}\rangle_{[\mu]} =
  \langle\monog{k,m}{-},\ \antimonog{k,m}{-}\rangle_{[\mu]} -
  \frac{\langle\monog{k,m}{+},\
    \antimonog{k,m}{+}\rangle_{[\mu]}}{\|\monog{k,m}{+}\|^2_{[\mu]}}\|\monog{k,m}{-}\|^2_{[\mu]}.
\]
Note that
\[ \langle\monog{k,m}{-},\ \antimonog{k,m}{-}\rangle_{[\mu]}
  =\langle\monog{k,m}{+},\ \antimonog{k,m}{+}\rangle_{[\mu]}
\]
and $\|\monog{k,m}{-}\|^2_{[\mu]}=\|\monog{k,m}{+}\|^2_{[\mu]}$,
when $m\neq0$. Therefore
$ \langle\ambi{k,m}{-+},\ambi{k,m}{--}\rangle_{[\mu]}=0$.
\end{proof}

It can shown, with more work, that $\monog{n,m}{\pm}$ are (up to
rescaling) the same polynomials defined in \cite{Morais2} (cf. \cite{MoNgK}); we will not
need this fact here.


\section{Spheroidal contragenic  functions\label{sec:contragenics}}

We now come to our main subject.  It is well known that every
$\C$-valued harmonic function in a simply connected domain in the
complex plane $\C$ is expressible as the sum of a holomorphic function
and an antiholomorphic function; these two elements are unique up to a
constant summand.  There are many generalizations of this fact for
monogenic functions on quaternions \cite{Sudbery} and Clifford
algebras \cite{Brackx}. A similar result for monogenic functions from
$\R^3\to\H$ is given in \cite{Cacao3}. However, it was
discovered in \cite{Alvarez} by a dimension count that the
corresponding statement for monogenic functions $\R^3\to\R^3$ does not
hold, due to the fact that the multiplication in $\R^3$ is not a
closed operation in $\H$. In other words, there are harmonic functions
which are not expressible as the sum of a monogenic and an
antimonogenic function.

We summarize the dimensions over $\R$ of the relevant spaces of
polynomials in Table \ref{tab:dimensions}. The subscript $*$ refers
to polynomials of degree at most $n$.
\begin{table}[ht!]
\[ \begin{array}{|c|c|}
  \hline
 \mbox{Space of polynomials} & \dim_\R  \\[.5ex] \hline
  \s \Har_*^{(n)}(\R) & (n+1)^2  \\[.5ex] \hline
  \s \Har_*^{(n)}(\R^3) & 3(n+1)^2 \\[.5ex] \hline
  \s \M_*^{(n)},\ \overline{\M}_*^{(n)} & (n+3)(n+1) \\[.5ex] \hline
 \s  \M_*^{(n)}\cap\overline{\M}_*^{(n)} & 2n+3 \\[.5ex] \hline
 \s  \M_*^{(n)}+\overline{\M}_*^{(n)}& 2(n^2+3n+1)+1 \\[.5ex] \hline
\end{array}
\]
  \caption{Dimensions of spaces of polynomials ($n\geq0$).}
  \label{tab:dimensions}
\end{table}


\subsection{Spheroidal contragenic polynomials}

One way to quantify the failure of a harmonic function to be ambigenic
is via orthogonal complements. We will write
\[ \M_2(\Omega)=\M(\Omega)\cap L_2(\Omega).
\]
Since scalar-valued (i.e.\ $\R e_0$-valued) functions are by
definition orthogonal in $L_2(\Omega)$ to functions which take values in
$\R e_1+\R e_2$, there is a natural orthogonal direct sum
decomposition of the space of square-integrable ambigenic functions,
namely
\begin{equation}\label{eq:ambigenicspace}
\M_2(\Omega) + \overline{\M}_2(\Omega) =
    \Sc\M_2(\Omega) \oplus \Vec\M_2(\Omega).
  \end{equation}

In any domain $\Omega$, a harmonic function
$h\in\Har(\Omega)\cap L_2(\Omega)$ is called $\Omega$-\textit{contragenic} when
it is orthogonal to all square-integrable ambigenic functions, that
is, if it lies in
\[ \N(\Omega)=(\M_2(\Omega)+\overline{\M}_2(\Omega))^\bot,
\]
where the orthogonal complement is taken in
$\Har(\Omega)\cap L_2(\Omega)$. Let
$\N^{(n)}(\Omega)\subset \N(\Omega)$ denote the subspace of
contragenic polynomials of degree $n$, and let
$\N_*^{(n)}(\Omega)\subset \N(\Omega)$ be the subspace of polynomials of degree
$\le n$. Unlike the spaces of harmonic, monogenic, antimonogenic and
ambigenic polynomials, the definition of $\N^{(n)}(\Omega)$ and
$\N_*^{(n)}(\Omega)$ involves the $L_2$ inner product and thus
depends on $\Omega$.

We now return to spheroids, and write
$\N_*^{(n)}[\mu]=\N_*^{(n)}(\Omega_\mu)$.  Let $n\geq1$. In
\cite{Alvarez}, it was proved that for $n\ge1$, the homogeneous
polynomials of degree $n$ which are contragenic on the sphere
$\Omega_0$ form a space of dimension $2n-1$.  It was also observed
when an $\R$-valued harmonic homogeneous polynomial is completed as
the scalar part of a monogenic function (unique up to adding a
monogenic constant), the vector part can also be taken to be a
homogeneous polynomial of the same degree.

Since the spheroidal harmonics and monogenics are not homogeneous, it
is preferable to combine the dimensions up to $n$;
for the sphere we have
$\dim \N_*^{(n)}[0]=n^2$.  Since the dimension of an orthogonal
complement within a fixed vector space does not depend on the inner
product used, and since the harmonic and the ambigenic polynomials of
degree $\le n$ do not depend on the domain, it is clear that we have
in general
\[ \dim \N_*^{(n)}[\mu]=n^2.
\]

We now give an explicit construction of a basis of the $\N_*^{(n)}$,
using as building blocks the components of the monogenic functions.
Write
\begin{equation}\label{eq:coefiv}
  \coefiv{n,m}[\mu] =
  \frac{ \|\harbase{n,m+1}{+}[\mu]\|^2_{[\mu]}}
       {(n+m+1)(n+m+2)^2 \|\harbase{n,m-1}{+}[\mu] \|^2_{[\mu]}}.
\end{equation}

\begin{defi} Let $n\ge1$.   The
\textit{basic contragenic polynomials} for $\Omega_\mu$ are
\begin{align*}
   \contra{n,m}{\pm}[\mu] =& \frac{\coefiv{n,m}}{n+m+1}
    \big( \Vec\monog{n,m}{\mp}[\mu] \mp \monog{n,m}{\pm}[\mu]e_3 \big)
    \pm [\monog{n,m}{\pm}[\mu]]_0e_3 ]\\
     & +  \big(-\Vec \monog{n,m}{\mp}[\mu]
    \mp\monog{n,m}{\pm}[\mu]e_3 \pm [\monog{n,m}{\pm}[\mu]]_0 e_3 \big)
\end{align*}
for  $1\leq m\leq n-1$, while for $m=0$ we define
\[ \contra{n,0}{}[\mu] =
\frac{1}{n+2}\left(\harbase{n,1}{-}[\mu]e_1-\harbase{n,1}{+}[\mu]e_2\right).
\]
\end{defi}

Examples are given in Table \ref{tab:firstcontragenics}.
In what follows, we will continue to write $\monog{n,m}{\pm}$,
$\contra{n,m}{\pm}$, $\coefiv{n,m}$ in place of
$\monog{n,m}{\pm}[\mu]$, $\contra{n,m}{\pm}[\mu]$,
$\coefiv{n,m}[\mu]$ when $\mu$ is fixed. The basic contragenic
polynomials may be expressed in terms of their $e_1$, $e_2$ components
by defining
\begin{equation}\label{eq:Anm}
   \coefv{n,m}{\pm}=\frac{\coefiv{n,m} \pm (n+m+1)}{n+m+1};
\end{equation}
then
\begin{align}
  \contra{n,m}{\pm}
    =&  (\coefv{n,m}{-}[\monog{n,m}{\mp}]_1
   \mp \coefv{n,m}{+} [\monog{n,m}{\pm}]_2 )e_1  \nonumber\\
  & \quad + (\coefv{n,m}{-}\left[\monog{n,m}{\mp}\right]_2\pm
  \coefv{n,m}{+}\left[\monog{n,m}{\pm}\right]_1)e_2 .
 \end{align}

  Also, from
\[
\monog{n,m}{\pm}[\mu]e_3=\left[\monog{n,m}{\pm}[\mu]\right]_2e_1-\left[\monog{n,m}{\pm}[\mu]\right]_1e_2+\left[\monog{n,m}{\pm}[\mu]\right]_0e_3.
\]
one sees that
\begin{align}\label{eq:contracoords}
  \contra{n,m}{\pm} =&
  \big(\coefiv{n,m}\harbase{n,m-1}{\mp}
    +\frac{1}{n+m+2}\harbase{n,m+1}{\mp}\big)e_1 \nonumber\\
  &\pm \big( \coefiv{n,m} \harbase{n,m-1}{\pm} -
     \frac{1}{n+m+2} \harbase{n,m+1}{\pm} \big)e_2.
\end{align}
 
\begin{table}[htbp!]
\[ \begin{array}{|c|c|l|}
\hline
n & m & \s\contra{n,m}{\pm}\\
\hline
1 & 0 & \s\contra{1,0}{}=-x_2e_1+x_1e_2\\[.5ex]
\cline{1-3}
\multirow{7}{*}{2} & 0 & \contra{2,0}{}=-3x_0x_2e_1+3x_0x_1e_2\\[.5ex]
\cline{2-3}
& 1 & \st\contra{2,1}{+}=6x_1x_2e_1+\dfrac{3}{30-20\mu^2+6\mu^4}\big(25x_2^2-2\mu^2-10x_2^2\mu^2\\
& &\qquad\quad +4\mu^4+x_2^2\mu^4-2\mu^6+ 10x_0^2(-1+\mu^2)^2\\
& &\qquad\quad  +x_1^2(-35+30\mu^2-11\mu^4)\big)e_2\\
& & \s\contra{2,1}{-}=\dfrac{3}{30-20\mu^2+6\mu^4}\big(-35x_2^2-2\mu^2+30x_2^2\mu^2+4\mu^4\\
& &\qquad\quad  -11x_2^2\mu^4-2\mu^6+x_1^2(-5+\mu^2)^2\\
& &\qquad\quad  +10x_0^2(-1+\mu^2)^2\big)e_1+6x_1x_2e_2\\[.5ex]
\cline{1-3}
\multirow{16}{*}{3} & & \s\contra{3,0}{}=\frac{3}{14}x_2(-28x_0^2+7x_1^2+7x_2^2+4\mu^2)e_1\\
& 0 &\qquad\quad  -\frac{3}{14}x_1(-28x_0^2+7x_1^2+7x_2^2+4\mu^2)e_2\\[.5ex]
\cline{2-3}
& & \st\contra{3,1}{+}=30x_0x_1x_2e_1+\dfrac{15x_0}{70-84\mu^2+30\mu^4}\big(49x_2^2\\
& &\qquad\quad  -6\mu^2-42x_2^2\mu^2+12\mu^4+9x_2^2\mu^4-6\mu^6+14x_0^2(-1+\mu^2)^2\\
& 1 &\qquad\quad  +x_1^2(-91+126\mu^2-51\mu^4)\big)e_2\\
& & \s\contra{3,1}{-}=\dfrac{15x_0}{70-84\mu^2+30\mu^4}\big(-91x_2^2-6\mu^2\\
& &\qquad\quad  +126x_2^2\mu^2+12\mu^4-51x_2^2\mu^4-6\mu^6+x_1^2(7-3\mu^2)^2\\
& &\qquad\quad  +14x_0^2(-1+\mu^2)^2\big)e_1+30x_0x_1x_2e_2\\[.5ex]
\cline{2-3}
& & \st\contra{3,2}{+}=-\dfrac{30x_2}{35-14\mu^2+3\mu^4}\big(-21x_2^2-2\mu^2+14x_2^2\mu^2+4\mu^4\\
& &\qquad\quad  -5x_2^2\mu^4-2\mu^6+x_1^2(-7+\mu^2)^2+14x_0^2(-1+\mu^2)^2\big)e_1\\
& &\qquad\quad  -\dfrac{30x_1}{35-14\mu^2+3\mu^4}\big(49x_2^2-2\mu^2-14x_2^2\mu^2+4\mu^4+x_2^2\mu^4\\
& &\qquad\quad  -2\mu^6+14x_0^2(-1+\mu^2)^2+x_1^2(-21+14\mu^2-5\mu^4)\big)e_2\\
& 2 & \s\contra{3,2}{-}=\dfrac{60x_1}{35-14\mu^2+3\mu^4}\big(28x_2^2+\mu^2-14x_2^2\mu^2-2\mu^4+4x_2^2\mu^4\\ & &\qquad\quad  +\mu^6-7x_0^2(-1+\mu^2)^2+x_1^2(-7+\mu^4)\big)e_1\\
& &\qquad\quad  -\dfrac{60x_2}{35-14\mu^2+3\mu^4}\big(-7x_2^2+\mu^2-2\mu^4+x_2^2\mu^4\\
& &\qquad\quad  +\mu^6-7x_0^2(-1+\mu^2)^2+2x_1^2(14-7\mu^2+2\mu^4)\big)e_2\\[.5ex]
 \hline
\end{array} \]
\caption{Spheroidal Contragenic Polynomials of Low Degree.}
  \label{tab:firstcontragenics}
  \end{table}

\begin{theo}\label{th:contragenicbasis}
  The $n^2$ functions
  $\{\contra{k,m}{\pm}[\mu], \contra{k,0}{}[\mu]\}$
  ($1\leq k\leq n$, $1\leq m\leq k-1$) form an orthogonal basis for
  $\N_*^{(n)}[\mu]$.
\end{theo}
\begin{proof}
  First we prove that $\contra{k,0}{}[\mu]$ and
  $\contra{k,m}{\pm}[\mu]$ are indeed contragenic. As they have no
  scalar parts, it suffices to show that they are orthogonal to
  $\Vec\M_*^{(n)}$. To do this, we use the basis obtained by dropping
  the scalar parts of the basis for $\M_*^{(n)}$ given in Theorem
  \ref{th:monogbase}. Since
\[ \{\Phi_{m_1}^+,\ \Phi_{m_2}^-|\ m_1\geq0,m_2\geq1\}
\]
is a system of orthogonal functions in $[0,\pi]$, then
when $1\leq m_1\leq k_1$ and $1\leq m_2\leq k_2$,
it is clear that
\[  \langle\contra{k_1,m_1}{+}, \, \Vec\monog{k_2,m_2}{+} \rangle_{[\mu]}
  =\langle\contra{k_1,m_1}{-}, \, \Vec\monog{k_2,m_2}{-} \rangle_{[\mu]}=0.
\]
On the other hand, when $m_1>0$ and $m_2\geq0$,  we have that
\begin{align*}
 \langle\contra{k_1,m_1}{\pm}, \,
   \Vec \monog{k_2,m_2}{\mp} \rangle_{[\mu]}
= &  \coefv{k_1,m_1}{-}\int_{\Omega_\mu}\left[\monog{k_1,m_1}{\mp}\right]_1
 \left[\monog{k_2,m_2}{\mp}\right]_1 \,dx\\
& \mp \coefv{k_1,m_1}{+}\int_{\Omega_\mu}\left[\monog{k_1,m_1}{\pm}\right]_2
   \left[\monog{k_2,m_2}{\mp}\right]_1 \,dx\\
&  +  \coefv{k_1,m_1}{-}\int_{\Omega_\mu}\left[\monog{k_1,m_1}{\mp}\right]_2
   \left[\monog{k_2,m_2}{\mp}\right]_2 \,dx\\
& \pm \coefv{k_2,m_2}{+}\int_{\Omega_\mu}\left[\monog{k_1,m_1}{\pm}\right]_1
   \left[\monog{k_2,m_2}{\mp}\right]_2 \,dx
\end{align*}
where $\coefv{k,m}{\pm}$ come from \eqref{eq:Anm}.
Since the system
\[ \left\{\Vec\monog{k,m}{+},\ \Vec\monog{j,l}{-}|\
 0\leq k\leq n,\ 0\leq m\leq k,\ 1\leq j\leq n,\ 1\leq l\leq j\right\}
\]
is orthogonal, straightforward computations show that
\begin{align}
 \langle\contra{k_1,m_1}{\pm},\,
 \Vec\monog{k_2,m_2}{\mp}\rangle_{[\mu]}
   =& \frac{\pi}{2}\bigg( 2\coefiv{k_1,m_1}(k_1+m_1+1)\,
  \int_{0}^{1/\mu}\int_{0}^{\pi}(\harbaseq{k_1,m_1-1})^2\,du\,dv \nonumber\\
&- \frac{2}{(k_1+m_1+2)^2}\int_{0}^{1/\mu}\,
  \int_{0}^{\pi}(\harbaseq{k_1,m_1+1})^2\,du\,dv \nonumber\\
&\mp \frac{2\delta_{0,m_1}}{(k_1+2)^2}\,
   \int_{0}^{1/\mu}\int_{0}^{\pi}(\harbaseq{k_1,1})^2\,du\,dv
  \bigg)\delta_{m_1,m_2}\delta_{k_1,k_2}.
\end{align}

Furthermore, using the expression \eqref{eq:contracoords} and recalling that
\begin{align*}
  \Vec  \monog{k,m}{-}  = &
  \dfrac{1}{2}\big[\big( (k+m+1)\harbase{k,m-1}{-}
  \frac{1}{k+m+2}\harbase{k,m+1}{-} \big) e_1\\
   & +  \big((k+m+1)\harbase{k,m-1}{+}
   + \frac{1}{k+m+2}\harbase{k,m+1}{+}\big)e_2\big]
\end{align*}
when $m>0$, we obtain that
\begin{align*}
  \langle \contra{k,m}{+},\ \Vec\monog{k,m}{-}\rangle_{[\mu]}
   = &  \frac{1}{2}\bigg[ \int_{\Omega_\mu} \big(
    \coefiv{k,m}\harbase{k,m-1}{-} +
    \frac{1}{k+m+2}\harbase{k,m+1}{-}\big) \\
& \times \big((k+m+1)\harbase{k,m-1}{-}
  -\frac{1}{k+m+2}\harbase{k,m+1}{-} \big) \,dx\\
& + \int_{\Omega_\mu} \big( \coefiv{k,m}\harbase{k,m-1}{+}
 -\frac{1}{k+m+2}\harbase{k,m+1}{+}\big)\\
& \times  \big((k+m+1)\harbase{k,m-1}{+}
   +\frac{1}{k+m+2}\harbase{k,m+1}{+} \big) \,dx \bigg]\\
 = & \coefiv{k,m}\|\harbase{k,m-1}{+}\|^2_{[\mu]}
 -\frac{1}{(k+m+2)^2}\|\harbase{k,m+1}{+}\|^2_{[\mu]}\\
 = & 0
\end{align*}
by \eqref{eq:coefiv}.
Similarly, the orthogonality of $\{\Phi_m^+,\Phi_l^-\}$ gives
$\langle\contra{k,m}{-}, \, \Vec\monog{k,m}{+} \rangle_{[\mu]}=0$.
Next, we expand
\begin{align*}
 \langle \contra{k_1,0}{},\Vec\monog{k_2,m}{\pm}\rangle_{[\mu]}
  = &   \frac{1}{2(k_1+2)} \bigg((k_2+m+1)\int_{\Omega_\mu}
   \harbase{k_1,1}{-}\harbase{k_2,m-1}{\pm} \,dx\\
& - \frac{1}{k_2+m+2}\int_{\Omega_\mu}
       \harbase{k_1,1}{-}\harbase{k_2,m+1}{\pm} \,dx\\
& \pm \big( (k_2+m+1)\int_{\Omega_\mu}
      \harbase{k_1,1}{+}\harbase{k_2,m-1}{\mp} \,dx\\\
& \quad +  \frac{1}{k_2+m+2}\int_{\Omega_\mu}
    \harbase{k_1,1}{+}\harbase{k_2,m+1}{\mp} \,dx  \big) \bigg) \\
& = 0
\end{align*}
again by orthogonality of $\{\Phi_m^+,\Phi_l^-\}$.
For $k_1\neq k_2$, by the orthogonality of the system
\begin{align*}
\bigg\{\harbase{k_1,m_1}{+},\harbase{k_2,m_2}{-}| \ & 0\leq k_1\leq n_1,\
0\leq k_2\leq n_2,\\
 & 0\leq m_1\leq k_1,\ 1\leq m_2\leq k_2,\\
 & n_1,n_2\geq0\bigg\}
\end{align*}
it remains to check that
\begin{align*}
  \langle\contra{k,0}{},\, \Vec\monog{k,2}{-}\rangle_{[\mu]}
 = &\frac{(k+m+1)}{2(k+2)}\big(\int_{\Omega_\mu}
  \left(\harbase{k,1}{-}\right)^2 \,dx-\int_{\Omega_\mu}
  \left(\harbase{k,1}{+}\right)^2 \,dx\big) \\
 = & 0 ;
\end{align*}
the last equality is a consequence of
\begin{align*}
   \int_{\Omega_\mu}(\harbase{k,1}{-})^2 \,dx
  & = \int_{0}^{\pi} \int_{0}^{1/\mu}(\harbaseq{k,1})^2
       \,dv\,du\int_{0}^{2\pi}\sin^2\phi d\phi\\
  & =  \int_{0}^{\pi} \int_{0}^{1/\mu}(\harbaseq{k,1})^2
    \,dv\,du\int_{0}^{2\pi}\cos^2\phi \,d\phi\\
  & =  \int_{\Omega_\mu}(\harbase{k,1}{+})^2 \,dx.
\end{align*}
We have verified that the functions $\contra{k,m}{\pm}$ are contragenic.

It remains to prove the orthogonality of the system
$\{\contra{k_1,m}{\pm},\contra{k_2,0}{}\}$. Using the expression
\eqref{eq:contracoords}, when $1\leq m_1,m_2$ we have
\begin{align*}
\langle\contra{k_1,m_1}{\pm},\contra{k_2,m_2}{\pm}\rangle_{[\mu]}  = &
\coefiv{k_1,m_1}\coefiv{k_2,m_2}\int_{\Omega_\mu}\harbase{k_1,m_1-1}{\mp}\harbase{k_2,m_2-1}{\mp} \,dx\\
& +
\frac{\coefiv{k_1,m_1}}{k_2+m_2+2}\int_{\Omega_\mu}\harbase{k_1,m_1-1}{\mp}\harbase{k_2,m_2+1}{\mp} \,dx\\
& +
\frac{\coefiv{k_2,m_2}}{k_1+m_1+2}\int_{\Omega_\mu}\harbase{k_1,m_1+1}{\mp}\harbase{k_2,m_2-1}{\mp} \,dx\\
& +
\frac{1}{(k_1+m_1+2)(k_2+m_2+2)}\int_{\Omega_\mu}\harbase{k_1,m_1+1}{\mp}\harbase{k_2,m_2+1}{\mp} \,dx \\
& +
\coefiv{k_1,m_1}\coefiv{k_2,m_2}\int_{\Omega_\mu}\harbase{k_1,m_1-1}{\pm}\harbase{k_2,m_2-1}{\pm} \,dx\\
& -
\frac{\coefiv{k_1,m_1}}{k_2+m_2+2}\int_{\Omega_\mu}\harbase{k_1,m_1-1}{\pm}\harbase{k_2,m_2+1}{\pm} \,dx\\
& -
\frac{\coefiv{k_2,m_2}}{k_1+m_1+2}\int_{\Omega_\mu}\harbase{k_1,m_1+1}{\pm}\harbase{k_2,m_2-1}{\pm} \,dx\\
& +
\frac{1}{(k_1+m_1+2)(k_2+m_2+2)}\int_{\Omega_\mu}\harbase{k_1,m_1+1}{\pm}\harbase{k_2,m_2+1}{\pm} \,dx.
\end{align*}
Thus, by a similar argument to that used in the proof of Theorem
\ref{th:monogbase},
\begin{align*}
  \langle\contra{k_1,m_1}{\pm},\contra{k_2,m_2}{\pm}\rangle_{[\mu]}
  = & 2\pi\delta_{m_1,m_2}\delta_{k_1,k_2}\bigg(
  (\coefiv{k_1,m_1})^2\int_{0}^{\pi}\int_{0}^{1/\mu}
   (\harbaseq{k_1,m_1-1})^2dvdu \\
  & +  \big(\frac{1}{k_1+m_1+2}\big)^2\int_{0}^{\pi}\int_{0}^{1/\mu}
    (\harbaseq{k_1,m_1+1})^2 \,dv\,du \bigg).
\end{align*}
On the other hand, when $1\leq m\leq k_1$, we have that
\begin{align*}
  \langle\contra{k_1,0}{},\contra{k_2,m}{\pm}\rangle_{[\mu]}
  = &   \frac{1}{k_1+2} \bigg( \coefiv{k_2,m}\int_{\Omega_\mu}
  \harbase{k_1,1}{-}\harbase{k_2,m-1}{\mp}   \, \,dx \\
  & +  \frac{1}{k_2+m+2}\int_{\Omega_\mu}\harbase{k_1,1}{-}\harbase{k_2,m+1}{\mp}
\, \,dx\\
& \mp  \coefiv{k_2,m}\int_{\Omega_\mu}\harbase{k_1,1}{+}\harbase{k_2,m-1}{\pm}
\, \,dx\\
& \pm
 \frac{1}{(k_2+m+2)}\int_{\Omega_\mu}\harbase{k_1,1}{+}\harbase{k_2,m+1}{\pm}
\, \,dx \bigg).
\end{align*}
Then, it is clear that
$\langle\contra{k_1,0}{},\contra{k_2,m}{-}\rangle_{[\mu]}=0$. It remains to
check that
\[\langle\contra{k_1,0}{},\contra{k_2,m}{+}\rangle_{[\mu]}=0,\]
But this follows again from the formula for the cosine of a sum of angles
and  $\int_{0}^{2\pi}\Phi_m^+d\phi=0$.

Finally, by the orthogonality of the system $\{\Phi_m^\pm\}$,
\[ \langle\contra{k_1,m_1}{\pm},\contra{k_2,m_2}{\mp}\rangle_{[\mu]}=0.
\]
\end{proof}


\subsection{Further observations}

Subspaces analogous to the homogeneous polynomials are obtained by defining
$\Harhat^{(n)}(\Omega)$ to be the orthogonal component of $\Har_*^{(n-1)}(\Omega)$
in $\Har_*^{(n)}(\Omega)$, so we have an orthogonal decomposition
\[    \Har_2(\Omega) = \bigoplus_{n=0}^\infty \Harhat^{(n)}(\Omega).
\]
(For $\Omega=\Omega_0$ this is in fact the decomposition by spherical harmonics.)
Similarly, let $\widehat\N^{(n)}[\mu]$ be the orthogonal component of $\N_*^{(n-1)}[\mu]$
in $\N_*^{(n)}[\mu]$, so  $\N_*^{(n)}[\mu]=\bigoplus_{k=1}^n\widehat\N^{(k)}[\mu]$.
Thus
\begin{align}  \label{eq:decomphar}
  \Harhat^{(n)}(\Omega_\mu) = \widehat\M^{(n)}(\Omega_\mu)
   \oplus \overline{\widehat\M^{(n)}(\Omega_\mu)}
   \oplus \widehat\N^{(n)}[\mu]
\end{align}
where the monogenic part $\widehat\M^{(n)}(\Omega_\mu)$ is defined analogously.
 
Natural linear mappings from the space of spherical harmonics
$\Har^{(n)}(\Omega_0)$ to the spheroidal $\Har^{(n)}(\Omega_\mu)$ were
worked out in \cite{BBS}. These correspondences will be studied in the
context of monogenic functions in future work.

\begin{theo} \label{th:dense}
  The functions $\contra{k,m}{\pm}[\mu]$ span a dense set in
  $\N(\Omega_\mu)$. Therefore the functions
  $\ambi{k,m}{\pm\pm}$,  $\contra{k,m}{\pm}[\mu]$ form an orthogonal
  basis for $\Har_2(\Omega_\mu)$.
\end{theo}
\begin{proof}
  Let $Z\in\N(\Omega_\mu)$. Write $Z=\sum_{k=0}^{\infty}U_k$, where
  $U_k\in\Harhat^{(k)}(\Omega_\mu)$, and let $U_k=Y_k+Z_k$ be the
  decomposition into ambigenic and contragenic polynomials. Thus
  $Z=Y+\sum_1^\infty Z_k$ where $Y=\sum_0^\infty Y$ is both ambigenic
  and contragenic, i.e.\ $Y=0$. Hence
  $Z\in \bigoplus\widehat\N^{(k)}[\mu]$ as required.
\end{proof}

The orthogonal decomposition
$(\M_2(\Omega)+\overline{\M}_2(\Omega))\oplus \N(\Omega)$ justifies
the idea of referring to the ``ambigenic part'' or the ``monogenic
part'' of any harmonic function $\Omega\to\R^3$ (the latter being
determined up to an additive monogenic constant). Theorem
\ref{th:dense} provides a method of calculation of this part in the
case of spheroids $\Omega_\mu$, by obtaining the Fourier coefficients
as in any Hilbert space, and then discarding the contragenic and
antimonogenic terms.

As a consequence of the fact that the norm on $L_2(\Omega_\mu)$
depends upon $\mu$, the spaces $\N(\Omega_\mu)$ are distinct for
distinct values of $\mu$. It is easy to see that

\begin{prop} $\mu\not=\mu'$ implies $\N(\Omega_\mu)\not=\N(\Omega_{\mu'})$.
\end{prop}

Indeed,  consider the polynomials of degree 2, which are of the form
\[  a_0\contra{2,0}{}[\mu] + a_+\contra{2,1}{+}[\mu] + a_-\contra{2,1}{-}[\mu]
\]
for real $a_0$, $a_+$, $a_-$. From Table \ref{tab:firstcontragenics}
we see that the coefficients of $x_1x_2$ and $x_2^2$ are,
respectively, $6a_+$ and $(75a_+-105a_-)/(30-20\mu^2+6\mu^4)$. These
coefficients determine $a_+$ and $a_-$, and then $a_0$ is determined
by the coefficient of $x_0x_1$.  Such a polynomial determines the
value of $\mu\in\R^+\cup i\R^+$, and thus can be in only one space
$\N(\Omega_\mu)$.

The fact that the notion of contragenicity depends on the domain
implies that it is not a local property, in contrast to harmonicity
and monogenicity. In particular, any attempt to seek a condition on
the derivatives of a harmonic function to detect whether it is
monogenic or not is doomed to failure.  It is not known, however,
whether such a condition may exist associated to a fixed domain, such
as a sphere or spheroid.


\section{Acknowledgements}

The first author's work is supported by CONACyT Grant 1600594. The second author acknowledges financial support by the Asociaci\'on Mexicana de Cultura, A. C.


\end{document}